\documentclass{amsart}
\usepackage{amsfonts,amssymb,amscd,amsmath,enumerate,verbatim,calc}
\usepackage[all]{xy}

\newcommand{\m}{\mathfrak{m} }

\newcommand{\bP}{\mathbb{\partial}}

\newcommand{\rt}{\rightarrow}

\newcommand{\ov}{\overline}

\newcommand{\F}{\mathcal{F} }

\theoremstyle{plain}

\newtheorem{thm}{Theorem}

\newtheorem{theorem}{Theorem}[section]
\newtheorem{corollary}[theorem]{Corollary}
\newtheorem{lemma}[theorem]{Lemma}
\newtheorem{proposition}[theorem]{Proposition}

\theoremstyle{definition}

\newtheorem{s}[theorem]{}
\newtheorem{remark}[theorem]{Remark}

\theoremstyle{remark}

\numberwithin{equation}{theorem}

\begin{document}

\title[\MakeLowercase{de} Rham ]{\MakeLowercase{de} Rham  cohomology of $H^1_{(f)}(R)$ where $V(f)$ is a smooth hypersurface in $\mathbb{P}^n$}
 \author{Tony~J.~Puthenpurakal}
\author{Rakesh B. T. Reddy  }
\date{\today}
\address{Department of Mathematics, IIT Bombay, Powai, Mumbai 400 076}

\email{tputhen@math.iitb.ac.in} \email{rakesh@math.iitb.ac.in}

\date{\today}


\begin{abstract}
Let $K$ be a field of characteristic zero, $R = K[X_1,\ldots,X_n]$.  Let $A_n(K)  = K<X_1,\ldots,X_n, \partial_1, \ldots, \partial_n>$  be the $n^{th}$ Weyl algebra over $K$.  We consider the case when $R$ and $A_n(K)$ is graded by giving $\deg X_i = \omega_i $ and $\deg \partial_i = -\omega_i$  for $i =1,\ldots,n$ (here $\omega_i$ are positive integers).  Set $\omega = \sum_{k=1}^{n}\omega_k$. Let $I$ be a graded ideal in $R$. By a result due to Lyubeznik  the local cohomology modules $H^i_I(R)$ are holonomic $A_n(K)$-modules for each $i \geq 0$.
In this article we compute  the de Rham cohomology modules $H^j(\bP ; H^1_{(f)}(R))$
for $j \leq n-2$ when $V(f)$ is a smooth hypersurface in $\mathbb{P}^n$ (equivalently
$A = R/(f)$ is an isolated singularity).
\end{abstract}

\maketitle

\section*{Introduction}
Let $K$ be a field of characteristic zero and let $R = K[X_1,\ldots,X_n]$.
We consider $R$ graded with $\deg X_i  = \omega_i$ for $i = 1,\ldots, n$; here $\omega_i$ are positive integers.  Set
$\m = (X_1,\ldots,X_n)$.
Let $I$ be a graded ideal in $R$. The local cohomology modules $H^*_I(R)$ are clearly graded $R$-modules. Let
$A_n(K)  = K<X_1,\ldots,X_n, \partial_1, \ldots, \partial_n>$  be the $n^{th}$ Weyl algebra over $K$. By a result due to Lyubeznik, see \cite{L}, the local cohomology modules $H^i_I(R)$ are \textit{holonomic} $A_n(K)$-modules for each $i \geq 0$.  We can consider $A_n(K)$ graded by giving $\deg \partial_i = -\omega_i$ for
 $i = 1,\ldots, n$.

 Let $N$ be a graded left $A_n(K)$ module. Now $\bP = \partial_1,\ldots,\partial_n$ are pairwise commuting $K$-linear maps. So we can consider the de Rham complex
$K(\bP;N)$. Notice that the de Rham cohomology modules  $H^*(\bP;N)$ are in general only \textit{graded} $K$-vector spaces. They are finite dimensional if $N$ is holonomic; \cite[Chapter 1, Theorem 6.1]{BJ}.
 In particular $H^*(\bP;H^*_I(R))$ are finite dimensional graded $K$-vector spaces.
 By \cite[Theorem 1]{TJP} the de Rham cohomology modules $H^*(\bP ; H^*_I(R))$ is concentrated in degree
$- \omega$, i.e., $H^*(\bP ; H^*_I(R))_j = 0$
for $j \neq - \omega$.

Let $f$ be a homogenous polynomial in $R$ with $A = R/(f)$ an isolated singularity, i.e., $A_P$ is regular for all homogeneous prime ideals $P \neq \m$. Note that $V(f)$ is a smooth hypersurface in $\mathbb{P}^n$. The main result of this paper is:
\begin{thm}\label{main}
(with hypotheses as above). Then $H^i(\bP; H^1_{(f)}(R)) = 0$ for $i \leq n -2$ and $i \neq 1$. Also $H^1(\bP; H^1_{(f)}(R)) = K$.
\end{thm}

 By \cite[Theorem 2.7]{TJ} we have $H^0(\bP; H^1_{(f)}(R)) = 0$.
In this paper we extend a technique from \cite{TJP}. In that paper the first author related  $H^{n-1}(\bP;  H^1_{(f)}(R))$   with $H^{n-1}(\partial(f);A)$. In this paper quite generally  we prove that if $H^{i-1}(\partial(f);A) = 0$ then we
  construct a filtration  $\F = \{ \F_\nu \}_{\nu\geq 0}$ \quad consisting of $K$-subspaces of \quad $H^{i}(\bP; H^1_{(f)}(R))$  for $i \geq 1$ with $\F_\nu = H^{i}(\bP; H^1_{(f)}(R))$  for $\nu \gg 0$, $\F_\nu \supseteq F_{\nu-1}$ and $\F_0 = 0$ and $K$-linear maps
  $$\eta_\nu \colon \F_\nu/\F_{\nu-1} \rt H^i(\bP f; A)_{(\nu + n-i)\deg f - \omega}.$$
  We also show that $\eta_\nu$ is injective for $\nu \geq 2$. If $i \neq 1$ then $\eta_1$ is also injective. Furthermore  if $i = 1$ then $\ker( \eta_1) = K$. When $A$ is an isolated singularity then note that $H^i(\partial(f); A) = 0$ for $i \leq n-2$. This gives our result.

  We now describe in brief the contents of the paper. In section one we discuss a few preliminaries that we need. In section two we construct certain functions which we need to define $\eta_\nu$. In section three we construct our filtration of $H^i(\bP; H^1_{(f)}(R))$ and prove our result.

 \section{Preliminaries}
 In  this section  we discuss a few preliminary results that we need.
\begin{remark}
Although all the results are stated for de Rham cohomology of a $A_n(K)$-module $M$, we will actually work with
de Rham homology. Note that $H_i(\bP, M) = H^{n-i}(\bP, M)$ for any $A_n(K)$-module $M$. Let $S = K[\partial_1,\ldots,\partial_n]$. Consider it as a subring of $A_n(K)$. Then note that $H_i(\bP, M)$ is the $i^{th}$ Koszul homology module of $M$ with respect to $\bP$.
\end{remark}

\s Let $A$ be commutative ring and $a=a_1, \cdots, a_n \in A.$ Let
$I \subseteq \{1, \cdots , n\},$ $|I|=m.$ Say $I=\{i_1<i_2< \cdots
<i_m\}.$ Let
\[
e_I:=e_{i_1}\wedge e_{i_2}\wedge \cdots \wedge e_{i_m}.
\]
Then the Koszul complex of $A$ with respect to $a$ is
\[
\mathbb{K}(a; A):= 0\rightarrow \mathbb{K}_n
\stackrel{\phi_n}\longrightarrow
\mathbb{K}_{n-1}\stackrel{\phi_{n-1}}\longrightarrow \cdots
\rightarrow \mathbb{K}_1\stackrel{\phi_1}\longrightarrow
\mathbb{K}_0\rightarrow 0.
\]
Here $\mathbb{K}_m=\bigoplus_{|I|=m}Ae_I.$
\[\text{Let} \ \xi=\sum_{|I|=m}\xi_Ie_I \in
\mathbb{K}_m \ \text{ we write } \  \xi=(\xi_I \mid |I|=m ).\ \text{
For the map}\ \mathbb{K}_p\stackrel{\phi_p}\longrightarrow
\mathbb{K}_{p-1},
\]

say $\phi_p(\xi)=U.$ Write $U= (U_J \mid |J|=p-1)$. Then
\[
U_J=\sum_{i\not \in J}(-1)^{\sigma(J\cup \{i\})}\left (a_i\xi_{J\cup
\{i\}}\right ).
\]
Here $J=\{j_1<j_2<\cdots <j_{p-1}\}$ and
\begin{align*}
\sigma(J\cup \{i\}) & = \begin{cases}0, & \text{if} \ i<j_1 ; \\
p, & \text{if} \ i>j_{p-1}\\
r, & \text{if} \ j_{r}<i<j_{r+1}.
\end{cases}
\end{align*}

\s Let $f\in R$ be a homogeneous polynomial. We consider elements of $R^m_f$ as
column-vectors. For $x \in R^m_f$ we write it as $x =
(x_1,\ldots,x_m)^\prime$; here $\prime$ indicates transpose.

 \s Let $f\in R$ be a homogeneous   polynomial. Set
$\partial=\partial_1,\cdots ,
\partial_n.$ Consider the commutative subring
$S=K[\partial_1,\cdots ,
\partial_n]$ of $A_n(K)$. The de Rham complex on a holonomic module $N$ is just the Koszul complex $K(\partial; N)$ of $N$ with respect to
$S$. In particular when $N=R_f$ we have,
\begin{align*}
\mathbb{K}(\partial; R_f)= 0\rightarrow
\mathbb{K}_n\stackrel{\phi_n}\longrightarrow\mathbb{K}_{n-1}\rightarrow
\cdots\mathbb{K}_{p}\stackrel{\phi_p}\longrightarrow
\mathbb{K}_{p-1}\rightarrow \cdots \rightarrow
\mathbb{K}_1\stackrel{\phi_1}\longrightarrow\mathbb{K}_0 \rightarrow
0.
\end{align*}
\begin{align*}
 \text{Here} \quad \mathbb{K}_0=R_f \quad \text{and} \quad \mathbb{K}_p=\bigoplus_{|I|=p}R_f(\omega_{i_1}+ \cdots + \omega_{i_p}) \quad \quad \quad.
\end{align*}
The  maps
$\mathbb{K}_p\stackrel{\phi_p}\longrightarrow \mathbb{K}_{p-1},\ \text{say} \ \xi = (\xi_I \ \mid \ |I|=p)'.$ Then \[\phi_p(\xi)= \left (\sum_{i\not \in J}(-1)^{\sigma(J\cup \{i\})}\left (\frac{\partial}{\partial x_i}\xi_{J\cup
\{i\}}\right ) : \ |J| = p-1 \right ). \]
 \s  Let $f\in R$ be a homogeneous   polynomial. Set $A=R/(f) $, and $\partial f=\partial f/\partial x_1,
\cdots,
\partial f/\partial x_n. $ Consider the Koszul complex $\mathbb{K}'(\partial f ; A)$ on  $A$ with
respect to $\partial f$.
\begin{align*}
\mathbb{K}'(\partial f ; A)=
0\rightarrow\mathbb{K}'_{n}\stackrel{\psi_n}\longrightarrow
\mathbb{K}'_{n-1}\cdots \rightarrow
\mathbb{K}'_p\stackrel{\psi_p}\longrightarrow
\mathbb{K}'_{p-1}\rightarrow \cdots \rightarrow
\mathbb{K}'_1\stackrel{\psi_1}\longrightarrow\mathbb{K}'_0\rightarrow
0.
\end{align*}

\begin{align*}
\text{Here}\quad \mathbb{K}'_p=\bigoplus_{|I|=p}A(-p\deg f
+\omega_{i_1}+ \cdots + \omega_{i_p}).
\end{align*}

\s \label{single degree} By \cite[ Theorem 1]{TJP},  $H_i(\partial ;
R_f)_j=0$ for $j\not = \omega,$ where $\omega = \omega_1+\cdots
+\omega_n.$

 \s Let $\xi \in R_f^m \setminus R^m$.  The element
$(a_1/f^i,a_2/f^i,\ldots,a_m/f^i)^\prime$, with $a_j \in R$ for all
$j$, is said to be a \textit{normal form} of $\xi$ if
\begin{enumerate}
\item
$\xi = (a_1/f^i,a_2/f^i,\ldots,a_m/f^i)^\prime$.
\item
$f$ does not divide $a_j$ for some $j$.
\item
$i\geq 1$.
\end{enumerate}
It can be easily shown that normal form of $\xi$ exists and is
unique (see \cite[Proposition 5.1]{TJP} ).

\s Let $\xi \in  R_f^m$. We define $L(f)$ as follows.

\textit{Case 1:} $\xi \in R_f^m\setminus R^m$.

Let $(a_1/f^i,a_2/f^i,\ldots,a_m/f^i)^\prime$ be the normal form of
$\xi$. Set $L(\xi) = i$. Notice $L(\xi) \geq 1$ in this case.

\textit{Case 2:} $\xi \in R^m\setminus \{0\}$.

Set $L(\xi) = 0$.

\textit{Case 3:} $\xi = 0$.

Set $L(\xi) = -\infty$.

The following properties of the function $L$ can be easily verified.
\begin{proposition}\label{prop-l}
(with hypotheses as above) Let $\xi,\xi_1,\xi_2 \in R_f^m$ and
$\alpha, \alpha_1,\alpha_2 \in K$.
\begin{enumerate}[\rm(1)]
\item
If $L(\xi_1) < L(\xi_2)$ then $L(\xi_1 + \xi_2) = L(\xi_2)$.
\item
 If $L(\xi_1) = L(\xi_2)$ then $L(\xi_1 + \xi_2) \leq L(\xi_2)$.
\item
$L(\xi_1 + \xi_2) \leq \max\{ L(\xi_1) , L(\xi_2) \}.$
\item
 If $\alpha \in K^*$ then $L(\alpha \xi ) = L(\xi)$.
\item
$L(\alpha \xi ) \leq L(\xi)$ for all $\alpha \in K$.
\item
$L(\alpha_1\xi_1 + \alpha_2\xi_2) \leq \max\{ L(\xi_1) , L(\xi_2)
\}.$
\item
Let $\xi_1,\ldots,\xi_r \in R^m_f$ and let $\alpha_1,\ldots,\alpha_r
\in K$. Then
$$L\left(\sum_{j=1}^{r}\alpha_j\xi_j \right) \leq \max \{ L(\xi_1), L(\xi_2),\ldots,L(\xi_r) \}.$$
\end{enumerate}
\end{proposition}

\section{Construction of certain functions}
In this section we construct few functions.\\
We  define a function,  $\theta :Z_p(\partial; R_f)\backslash
R^{\binom{n}{p}}\longrightarrow H_p(\partial f ; A)
 $, as follows.\\
Let $\xi \in Z_p(\partial; R_f)\backslash R^{\binom{n}{p}}$ and let
$ (\xi_I/f^c \ \mid \  |I|=p)'$ be the normal form of $\xi.$ As
$\phi_p(\xi)=0.$ We have for every $J$ such that $|J|=p-1$,
\begin{align*}
\sum_{i\not \in J}(-1)^{\sigma(J\cup \{i\})}\frac{\partial}{\partial
x_i} \left ( \frac{\xi_{J\cup\{i\}}}{f^c} \right )=&0.\\
\text{This implies} \quad \sum_{i\not \in J}(-1)^{\sigma(J\cup
\{i\})} \left ( \frac{\frac{\partial}{\partial
x_i}(\xi_{J\cup\{i\}})}{f^c} -
c\xi_{J\cup\{i\}}\frac{\frac{\partial f}{\partial x_i}}{f^{c+1}} \right )=& 0.\\
\text{So} \quad  f.\sum_{i\not \in J}(-1)^{\sigma(J\cup \{i\})}\left
( \frac{\partial}{\partial x_i}(\xi_{J\cup\{i\}})\right )= c.
\sum_{i\not \in J}(-1)^{\sigma(J\cup \{i\})}\left ( \xi_{J\cup\{i\}}
\frac{\partial f}{\partial x_i} \right ).
\end{align*}
Thus $f$ divides $\sum_{i\not \in J}(-1)^{\sigma(J\cup \{i\})}\left
( \xi_{J\cup\{i\}}  \frac{\partial f}{\partial x_i} \right ).$
Therefore
\begin{align*}
(\bar{\xi_I} \ \mid \  |I|=p)'\in Z_p(\partial f ; A).
\end{align*}
Set $\theta(\xi)=[(\bar{\xi_I}\ \mid \  |I|=p)']\in H_p(\partial f ;
A).$\\
The following Lemma identifies the degree of $\theta(\xi)$.
\begin{lemma}\label{deg-theta} Assume $p < n.$ Let $\xi \in Z_p(\partial ; R_f)_{-\omega}$ be non zero.
Then\\
(a)  $\xi
\in  R^{\binom{n}{p}}_f\backslash R^{\binom{n}{p}}.$\\
(b)  If $L(\xi)=c.$ Then $\theta(\xi)\in H_p(\partial f ;
A)_{(c+p)\deg f-\omega}$.
\end{lemma}
\begin{proof}
$(a)$ Let $\xi=(\xi_I)'$ be non-zero in $Z_p(\partial ;
R_f)_{-\omega}$. Note that
\[
\xi \in \bigoplus_{|I|=p}(R_f(\omega_{i_1}+\cdots
+\omega_{i_p}))_{-\omega}. \] It follows that
\begin{align*}
\xi_I \in (R_f)_{-\omega + \sum^{p}_{s=1}\omega_{i_s}}.
\end{align*}
It follows that $\xi \in R^{\binom{n}{p}}_f\backslash
R^{\binom{n}{p}}.$

$(b)$ Let $(a_I/f^c \ \mid \ |I|=p)'$ be the normal form of $\xi.$
As $a_I/f^c \in R_f(\omega_{i_1}+\cdots + \omega_{i_p})_{-\omega}.$
It follows that
\begin{align*}
\deg (a_I)= c\deg f-\omega +\sum^p_{s=1}\omega_{i_s}.
\end{align*}
As $\theta(\xi)=[(\bar{a}_I \ \mid \  |I|=p)'].$ Let
\begin{align*}
\bar{a}_I \in A(-p\deg f+\sum^p_{s=1}\omega_{i_s})_t.
\end{align*}
Then
\begin{align*}
\bar{a}_I \in A_{(-p\deg f+\sum^p_{s=1}\omega_{i_s})+t}.
\end{align*}
It follows that
\begin{align*}
t=(c+p)\deg f-\omega.
\end{align*}
Thus $\theta(\xi)\in H_p(\partial f ; A)_{(c+p)\deg f-\omega}$.
\end{proof}
A natural condition we want in $\theta$ is that it vanishes on
boundaries. The following result gives a sufficient condition when
this happens.
\begin{proposition}\label{B_p=0}
Let $p < n$. Assume $H_{p+1}(\partial f ; A)=0$. Then $\theta
(B_p(\partial;R_f)_{-\omega}\backslash \{0\})$
$=0$.
\end{proposition}
\begin{proof}
Let $U\in B_p(\partial ; R_f)_{-\omega}$ be non zero. As
$B_p(\partial ; R_f)_{-\omega} \subset Z_p(\partial ;
R_f)_{-\omega}$. We get by Lemma \ref{deg-theta}(a) that $U\in
R_f^{\binom{n}{p}}\backslash R^{\binom{n}{p}}.$ 
 Set
  $$c = \min\{ \
j \ | \  j=L(\xi) \ \text{where} \ \phi_{p+1}(\xi)=U \ \text{and} \ \xi \in
(\mathbb{K}_{p+1})_{-\omega}  \}.$$

 Notice $c\geq 1$.
Let $\xi \in (\mathbb{K}_{p+1})_{-\omega}$ be such that $L(\xi)=c$
and $\phi_{p+1}(\xi)=U.$ Let $ (b_G/f^c \ \mid \  |G|=p+1)'$ be the
normal form of $\xi$. Let $U=(U_I \ \mid \  |I|=p)'$. Then
\begin{align*}
U_I =&\sum_{i\not \in I}(-1)^{\sigma(I\cup
\{i\})}\frac{\partial}{\partial x_i}
\left ( \frac{b_{I\cup\{i\}}}{f^c} \right )\\
=&\sum_{i\not \in I}(-1)^{\sigma(I\cup \{i\})} \left ( \frac{\frac{\partial}{\partial x_i}b_{I\cup\{i\}}}{f^c} -
c \frac{b_{I\cup\{i\}}\frac{\partial f}{\partial x_i}}{f^{c+1}}\right ) \\
 =&
\frac{f}{f^{c+1}}\sum_{i\not \in I}(-1)^{\sigma(I\cup
\{i\})}\left ( \frac{\partial}{\partial x_i}(b_{I\cup\{i\}})\right )\\
&-\frac{c}{f^{c+1}}\sum_{i\not \in I}(-1)^{\sigma(I\cup \{i\})}\left
( b_{I\cup\{i\}} \frac{\partial f}{\partial x_i}\right ).
\end{align*}
Set
\begin{align*}
  V_I=-c\sum_{i\not \in I}(-1)^{\sigma(I\cup \{i\})}\left (
b_{I\cup\{i\}} \frac{\partial f}{\partial x_i}\right ).
\end{align*}
Then
\begin{align*}
U_I=\frac{f*+V_I}{f^{c+1}}.
\end{align*}
\textit{Claim}: $f$ does not divides $V_I$ for some $I$ with
$|I|=p.$

First assume the claim. Then $U=(U_I \ \mid \ |I|=p)'=( \
(f*+V_I)/f^{c+1} \ \mid \  |I|=p \ )' $ is the normal form of $U$.
Therefore
\begin{align*}
\theta(U)=[(\bar{V_I})']=[\psi_{p+1}(-c\bar{b}_G \ \mid \
|G|=p+1)']=0.
\end{align*}

We now prove our claim. Suppose if possible $f|V_I$ for all $I$ with
$|I|=p$. Let $b = (b_G \ \mid \ |G| = p+1 )'.$ Then
\begin{align*}
\psi_{p+1}(-c\bar{b})= (\bar{V_I})'=(0,0,\cdots, 0)'.
\end{align*}
So $-c\bar{b}\in Z_{p+1}(\partial f; A).$ As $H_{p+1}(\partial f;
A)=0$, we get $-c\bar{b}\in B_{p+1}(\partial f; A).$ Thus
\begin{align*}
-c\bar{b}=\psi_{p+2}(\bar{r}),\quad \text{here} \quad r=(r_L\in R \
\mid \  |L|=p+2)
\end{align*}
For $p=n-1$ we have
\begin{align*}
-c\bar{b}&=0\\
\Rightarrow \quad cb_G&=\tilde{\alpha}_Gf \quad \text{for some }
\quad \tilde{\alpha}_G \in R.\\
\end{align*}
 Thus $f/b_G$ for all $G$. This contradicts the that
 $(b_G/f^c | \  |G| =p+1)'$ is the normal form of $\xi$. Therefore $p<n-1$. So
\begin{equation}\label{*}
-cb_G=\sum_{k\not \in G}(-1)^{\sigma(G\cup \{k\})}\left ( r_{G\cup
\{k\}}\frac{\partial f}{\partial x_k}\right ) + \tilde{\alpha}_Gf.
\end{equation}
Now we compute the degrees of $r_L$. Note that $\xi \in
(\mathbb{K}_{p+1})_{-\omega}.$ So
\begin{align*}
\frac{b_G}{f^c}\in
R_f(\omega_{i_1}+\cdots+\omega_{i_{p+1}})_{-\omega}.
\end{align*}
It follows that
\begin{equation}\label{**}
\deg b_G = c \deg f -\omega + \sum_{s=1}^{p+1}\omega_{i_s}.
\end{equation}
It can be easily checked that
\begin{align*}
\bar{b}_{I\cup\{i\}}\in A(-(p+1)\deg f + \omega_{i_1}+\cdots
+\omega_{i_{p+1}})_{(c+p+1)\deg f -\omega}.
\end{align*}
So
\begin{align*}
\bar{r}_L\in  A(-(p+2)\deg f + \omega_{i_1}+\cdots
+\omega_{i_{p+2}})_{(c+p+1)\deg f -\omega}.
\end{align*}
It follows that
\begin{equation}\label{***}
\deg r_L=(c-1)\deg f -\omega + \sum_{j=1}^{p+2}\omega_{i_j}.
\end{equation}
\textit{Case(1)}: Let $c=1$. Then by equation (\ref{*}) we get
$\tilde{\alpha}_G=0$. Also notice
\begin{align*}
\deg r_L = -\omega + \sum_{j=1}^{p+2}\omega_{i_j} < 0 \quad
\text{if} \quad n>p+2.
\end{align*}
So if $n>p+2$ we get $r_L=0$. So $b=0$ and $\xi =0$ a contradiction.

Now consider $n=p+2.$ Then $r_L=r_{12\cdots n}=$ constant. Say
$r_L=r$ . So by equation (\ref{*}) it follows
\begin{equation}\label{****}
b=(b_G \ \mid \  |G|=n-1)'=(-r\partial f/\partial x_1, \cdots,
(-1)^{n}r\partial f/\partial x_n)'.
\end{equation}
\[\text{As} \quad  U_I=\sum_{i\not \in I}(-1)^{\sigma(I\cup
\{i\})}\left ( \frac{\partial}{\partial x_i}(b_{I\cup\{i\}}/f)
\right ).\]
 Using equation(\ref{****}) we get $U_I=0$ for all $I$
with $|I|=p.$ Therefore $U=0$ a contradiction.

\textit{Case(2)}: Let $c\geq 2.$ By equation (\ref{*}) we have
\begin{align*}
\frac{-cb_G}{f^c}=\frac{1}{f^c}\sum_{k\not \in G}(-1)^{\sigma(G\cup
\{k\})} \left ( r_{G\cup \{k\}}\frac{\partial f}{\partial x_k}\right
) + \frac{\tilde{\alpha}_G}{f^{c-1}}.
\end{align*}
Notice
\begin{align*}
\frac{r_{G\cup \{k\}}\partial f/\partial
x_k}{f^c}=\frac{\partial}{\partial x_k}\left ( \frac{r_{G\cup
\{k\}}/(1-c)}{f^{c-1}} \right )-\frac{*}{f^{c-1}}.
\end{align*}
\[\text{Put} \quad \tilde{r}_{G \cup \{k\}}=\frac{r_{G \cup
\{k\}}}{c(c-1)}. \quad \text{ We obtain} \]
\begin{align*}
\frac{b_G}{f^c}=\sum_{k\not \in G}(-1)^{\sigma(G \cup \{k\})}\left (
\frac{\partial }{\partial x_k} \left ( \frac{\tilde{r}_{G \cup
\{k\}}}{f^{c-1}} \right ) \right ) + \frac{*}{f^{c-1}}.
\end{align*}
Set
\begin{align*}
 \delta = \left ( \frac{\tilde{r}_{G \cup \{k\}}}{f^{c-1}} \right )\quad \text{and}
 \quad
\tilde{\xi}= \left ( \frac{*}{f^{c-1}} \right ).
\end{align*}
 Then
\begin{align*}
\xi = \phi_{p+2}(\delta) + \tilde{\xi}.
\end{align*}
So we have $U=\phi_{p+1}(\xi)=\phi_{p+1}(\tilde{\xi})$ and $L(\xi)
\leq c-1$. This contradicts our choice of $c.$
\end{proof}

\section{Construction of a filtration on $H_p(\partial; R_f)$.}
In this section we construct a filtration of $H_p(\partial ; R_f).$
Throughout this section $1 \leq p < n$ and  $H_{p+1}(\partial f;
A)=0.$

 \s By \ref{single degree} we have
\begin{align*}
H_p(\partial ; R_f)=H_p(\partial ; R_f)_{-\omega}=\frac{Z_p(\partial
; R_f)_{-\omega}}{B_p(\partial ; R_f)_{-\omega}}.
\end{align*}

Let $x\in H_p(\partial ; R_f) $ be \textit{non-zero}. Define
\begin{align*}
L(x)=min\{ \  L(\xi) \  | \quad x=[\xi],\quad \text{where}\quad \xi
\in Z_p(\partial ; R_f)_{-\omega} \}.
\end{align*}

Let $\xi = (\xi_I/f^c \ \mid \ |I| = p )'\in Z_p(\partial ;
R_f)_{-\omega}$ be such that $x=[\xi].$ So $\xi \in
(\mathbb{K}_p)_{-\omega}.$ Thus $\xi \in
R^{\binom{n}{p}}_f(\omega_{i_1}+\cdots + \omega_{i_p})_{-\omega} $.
So if $\xi \not =0$ then $\xi \in R^{\binom{n}{p}}_f\backslash
R^{\binom{n}{p}}.$ It follows that $L(\xi)\geq 1.$ Thus $L(x)\geq
1.$ If $x=0$ set $L(x)= - \infty$.

    We now define a function
\begin{align*}
\tilde{\theta}: H_p(\partial; R_f)&\rightarrow H_p(\partial f ; A)\\
x &\rightarrow
\begin{cases} 0 & \text{if} \ x = 0 \\ \theta(\xi)& \text{if} \
x\not =0,x=[\xi],\quad\text{and}\quad L(x)=L(\xi).
\end{cases}
\end{align*}

\begin{proposition}
(with hypothesis as above ) $\tilde{\theta}(x)$ is independent of
$\xi.$
\end{proposition}
\begin{proof}
Suppose $x=[\xi_1]=[\xi_2] $ is non zero and
$L(x)=L(\xi_1)=L(\xi_2)=c.$ Let $(a_I/f^c)'$ be the normal form of
$\xi_1$ and $(b_I/f^c)'$ be the normal of $\xi_2.$ As
$[\xi_1]=[\xi_2]$ it follows that $\xi_1=\xi_2+\delta$ for some
$\delta \in B_p(\partial ; R_f)_{-w}.$ We get $j=L(\delta)\leq c$ by
\ref{prop-l}. Let $(c_I/f^j)'$ be the normal form of $\delta.$ We
consider two cases.

\textit{Case$(1)$}: $j<c.$ Then note that $a_I=b_I+f^{c-j}c_I$,  for
$|I|=p.$ It follows that
\begin{align*}
\theta(\xi_1)=[(\bar{a_I})']=[(\bar{b_I})']=\theta(\xi_2).
\end{align*}
\textit{Case$(2)$}: $j=c.$ Note that $a_I=b_I+c_I$ for $|I|=p.$ It
follows that
\begin{align*}
\theta(\xi_1)=\theta(\xi_2)+\theta(\delta).
\end{align*}
However by Proposition \ref{B_p=0} $\theta(\delta)=0.$ So
$\theta(\xi_1)=\theta(\xi_2).$ Hence $\tilde{\theta}(x)$ is
independent of choice of $\xi.$
\end{proof}
\s We now construct a filtration
$\mathcal{F}=\{\mathcal{F}_{\upsilon}\}_{\upsilon \geq 0}$ of
$H_p(\partial ; R_f).$ Set
\begin{align*}
\mathcal{F}_{\upsilon}=\{x \in H_p(\partial ; R_f)\quad | \quad
L(x)\leq \upsilon \}.
\end{align*}

\begin{proposition}
$(1)$ $\mathcal{F}_{\upsilon}$ is a $K-$subspace of $H_p(\partial ;
R_f).$\\
$(2)$  $\mathcal{F}_{\upsilon -1}\subseteq \mathcal{F}_{\upsilon} $
for all $\upsilon \geq 1.$\\
$(3)$ $\mathcal{F}_{\upsilon}=H_p(\partial ; R_f)$ for all $\upsilon
\gg 0.$\\
$(4)$ $\mathcal{F}_0 = 0.$
\end{proposition}
\begin{proof}
$(1) $  Let $x\in \mathcal{F}_{\upsilon}$ and let $\alpha \in K.$
Let $x=[\xi]$ with $L(x)=L(\xi) \leq \upsilon.$ Then $\alpha
x=[\alpha \xi]$. So
\begin{align*}
L(\alpha x)\leq L(\alpha \xi)\leq \upsilon.
\end{align*}
So $\alpha x \in \mathcal{F}_{\upsilon}$.

Let $x, x' \in \mathcal{F}_{\upsilon}$ be non-zero. Let $\xi, \xi'
\in Z_p(\partial ; R_f)$ be such that $x=[\xi],$ $x'=[\xi']$ and
$L(x)=L(\xi),$ $L(x')=L(\xi').$ Then $x+x'=[\xi + \xi'].$ It follows
that
\begin{align*}
L(x+x')\leq L(\xi+\xi')\leq max\{L(\xi), L(\xi')  \}\leq \upsilon.
\end{align*}
Thus $x+x' \in \mathcal{F}_{\upsilon}$.

$(2)$ This is clear from the definition.

$(3)$ Let $\mathcal{B}=\{ x_1, \cdots, x_l \}$ be a $K-$ basis of
$H_p(\partial ; R_f)=H_p(\partial ; R_f)_{-\omega}$. Let
\begin{align*}
c=max\{L(x_i) \quad | \quad i=1,\cdots, l\}.
\end{align*}
We claim that
\begin{align*}
\mathcal{F}_{\upsilon}= H_p(\partial ; R_f) \quad \text{for
all}\quad \upsilon \geq c.
\end{align*}
Fix $\upsilon \geq c.$ Let $\xi_i\in Z_p(\partial ; R_f)_{-\omega}$
be such that $x_i=[\xi_i]$ and $L(x_i)= L(\xi_i)$ for $i=1,\cdots,
l$.

Let $u\in H_p(\partial ; R_f).$ Say $u=\sum^l_{i=1}\alpha_ix_i$ for
some $\alpha_1,\cdots , \alpha_l \in K$. Then
$u=[\sum^l_{i=1}\alpha_i\xi_i]$. It follows that
\begin{align*}
L(u)\leq L(\sum^l_{i=1}\alpha_i \xi_i)\leq max\{ L(\xi)\quad | \quad
i=1,\cdots , l \}=c\leq \upsilon.
\end{align*}
So $u\in \mathcal{F}_{\upsilon}$. Hence
$\mathcal{F}_{\upsilon}=H_p(\partial ; R_f)$.

$(4)$ If $x\in H_p(\partial ; R_f)$ is non-zero then $L(x)\geq 1.$
Therefore $\mathcal{F}_0=0$.
\end{proof}
\s Let $\mathcal{G}=\bigoplus_{\upsilon\geq
1}\mathcal{F}_{\upsilon}/\mathcal{F}_{\upsilon-1}$. For
$\upsilon\geq 1$ we define
\begin{align*}
\eta_{\upsilon}:\frac{\mathcal{F}_{\upsilon}}{\mathcal{F}_{\upsilon-1}}&
\rightarrow H_p(\partial f ; A)_{(\upsilon + p)\deg f - \omega}\\
u & \rightarrow
\begin{cases} 0 & \text{if} \ u = 0 \\ \tilde{\theta}(x)& \text{if} \ u=x+ \mathcal{F}_{\upsilon -1} \quad \text{is non-zero}.
\end{cases}
\end{align*}

\begin{proposition}
(with hypothesis as above) $ \eta_{\upsilon}(u)$ is independent of
choice of $x$.
\end{proposition}
\begin{proof}
Suppose $u= x+\mathcal{F}_{\upsilon-1}=x'+ \mathcal{F}_{\upsilon-1}$
be non-zero. Then $x=x'+y$ where $y\in \mathcal{F}_{\upsilon-1}$. As
$u\not = 0$ we have $x, x' \in \mathcal{F}_{\upsilon} \backslash
\mathcal{F}_{\upsilon-1}$. So $L(x)=L(x')=\upsilon$. Say $x=[\xi]$,
$x'=[\xi']$, and $y=[\delta]$ where $\xi, \xi', \delta \in
Z_p(\partial ; R_f)_{-\omega}$ with $L(\xi)=L(\xi')=\upsilon$ and
$L(\delta)=L(y)=k\leq \upsilon-1$. So we have $\xi= \xi'+ \delta +
\alpha$ where $\alpha \in B_p(\partial ; R_f)_{-\omega}$. Let
$L(\alpha)=r.$ Not that $r\leq \upsilon$.

Let $(a_I/f^{\upsilon})',$ $(a'_I/f^{\upsilon})',$ $(b_I/f^k)'$ and
$(c_I/f^r)'$ be normal forms of $\xi, \xi',\delta$ and $\alpha$
respectively, here $|I|=p.$ So we have
\begin{align*}
a_I=a'_I+ f^{\upsilon-k}b_I + f^{\upsilon-r}c_I\quad \text{with}
\quad |I|=p.
\end{align*}

Case$(1)$: $r<\upsilon$. In this case we have that
$\bar{a}_I=\bar{a'}_I$ in $A$. So $\theta(\xi)=\theta(\xi')$. Thus
$\tilde{\theta}(x)=\tilde{\theta}(x')$.

Case$(2)$: $r=\upsilon$. In this case we have $\bar{a}_I=\bar{a'}_I
+ \bar{c}_I$ in $A$. So $\theta{(\xi)}=\theta{(\xi')}+
\theta{(\alpha)} $. However $\theta{(\alpha)}=0$ as $\alpha \in
B_p(\partial ; R_f)_{-\omega}$. Thus
$\tilde{\theta}(x)=\tilde{\theta}(x').$
\end{proof}

\begin{proposition}
(with notation as above). For all $\upsilon \geq 1$,
$\eta_{\upsilon}$ is $K-$linear.
\end{proposition}
\begin{proof}
Let $u, u' \in \mathcal{F}_{\upsilon}/\mathcal{F}_{\upsilon-1}$. We
first show that $\eta_{\upsilon}(\alpha u)=\alpha
\eta_{\upsilon}(u)$. If $\alpha=0$ or $u=0$ we have nothing to show.
So assume $\alpha \not =0$ and $u\not =0$. Say $u=x+
\mathcal{F}_{\upsilon-1}.$ Then $\alpha u= \alpha x +
\mathcal{F}_{\upsilon-1}.$ It can be easily shown that $\tilde{\theta}(\alpha x)=\alpha
\tilde{\theta}(x)$. So we get the result.

Next we show that $\eta_{\upsilon}(u+u')= \eta_{\upsilon}(u)+
\eta_{\upsilon}(u')$. We have nothing to show if $u$ or $u'$ is
zero. Now consider the case when $u+u'=0$. Then $u=-u'$. So
$\eta_{\upsilon}(u)=-\eta_{\upsilon}(u')$. Thus in this case
\begin{align*}
\eta_{\upsilon}(u+u')=0=\eta_{\upsilon}(u)+\eta_{\upsilon}(u').
\end{align*}
Now consider the case when $u,u'$ are non-zero and $u+u'$ non-zero.
Say $u=x+ \mathcal{F}_{\upsilon-1}$ and $u'= x'+
\mathcal{F}_{\upsilon-1}$. Note that as $u+u'$ is non-zero $x+x'\in
\mathcal{F}_{\upsilon}\backslash \mathcal{F}_{\upsilon-1}$. Let
$x=[\xi]$ and $x'=[\xi']$ where $\xi, \xi' \in Z_p(\partial;
R_f)_{-\omega}$ and $L(\xi)= \ L(\xi')=\upsilon$. Then
$x+x'=[\xi+\xi']$. Note that $L(\xi + \xi')\leq \upsilon$. But
$L(x+x')= \upsilon.$ So $L(\xi + \xi') = \upsilon$. Let
$(a_I/f^{\upsilon})',$ $(a'_I/f^{\upsilon})'$ be the normal forms of
$\xi,$ and $\xi'$ respectively. Note that
$((a_I+a'_I)/f^{\upsilon})'$ is the normal form of $\xi+ \xi'$. It
follows that $\theta(\xi+\xi')= \theta(\xi)+ \theta(\xi')$. Thus
$\tilde{\theta}(x+x')=\tilde{\theta}(x)+\tilde{\theta}(x')$.
Therefore
\begin{align*}
\eta_{\upsilon}(u+u')= \eta_{\upsilon}(u)+\eta_{\upsilon}(u').
\end{align*}
\end{proof}
Surprisingly the following result holds.
\begin{proposition}(with notation as above).\\
(a) $\eta_{\upsilon}$ is injective for all $\upsilon \geq 2$.\\
(b) If $p\not= n-1.$ Then $\eta_1$ also injective.\\
(c) If $p=n-1.$ Then $\ker(\eta_1)= K.$
\end{proposition}
\begin{proof}
Suppose if possible $\eta_\nu$ is not injective. Then there exists
non-zero $u \in \mathcal{F}_{\nu}/ \mathcal{F}_{\nu-1}$ with
$\eta_\nu(u) = 0$. Say $u = x + \mathcal{F}_{\nu -1}$. Also let $x =
[\xi] $ where $\xi \in Z_p(\partial; R_f)_{-\omega}$ and $L(\xi) =
L(x) = \nu$. Let $(a_I/f^{\upsilon} \ \mid \  |I|=p)'$ be the normal
form of $\xi$. So we have
\[
0 = \eta_\nu(u) =  \widetilde{\theta}(x)  = \theta(\xi) =
[(\ov{a_I})'].
\]
It follows that $(\ov{a_I})' = \psi_{p+1}(\ov{b})$, where $\ov{b} =
( b_G \ \mid \  |G|=p+1)'$. It follows that
\[
\ov{a_I} = \sum_{i\not \in I} (-1)^{\sigma(I\cup \{i\})}\left
(\bar{b}_{I\cup \{i\}}\bar{\frac{\partial f}{\partial x_i}}\right ).
\]

It follows that for $|I|=p$   we have the following equation in $R$:
\begin{equation}\label{t2-eq1}
a_I  = \sum_{i\not \in I} (-1)^{\sigma(I\cup \{i\})} \left (
b_{I\cup \{i\}}\frac{\partial f}{\partial x_i}\right )   + d_If,
\end{equation}
for some $d_I \in R$. Note that the above equation is of homogeneous
elements in $R$. So we have the following
\begin{equation}\label{t2-eq2}
\frac{a_I}{f^{\upsilon}}=\frac{\sum_{i\not \in I} (-1)^{\sigma(I\cup
\{i\})}b_{I\cup \{i\}}\frac{\partial f}{\partial x_i}}{f^{\upsilon}}
+ \frac{d_I}{f^{\upsilon -1}}.
\end{equation}
We consider two cases:

(a): Let $\nu \geq 2$. Set $\widetilde{b}_{I\cup \{i\}} =  -b_{I\cup
\{i\}}/(\upsilon-1)$. Then note that
\[
\frac{ b_{I\cup \{i\}} \frac{\partial f}{\partial x_i}}{f^\nu} =
\frac{\partial}{\partial x_i} \left( \frac{\widetilde{b}_{I\cup
\{i\}}}{f^{\nu-1}}\right)    - \frac{*}{f^{\nu-1}}.
\]

By equation (\ref{t2-eq2}) we have
\[
\frac{a_I}{f^\nu} = \sum_{i\not \in I} (-1)^{\sigma(I\cup \{i\})}
\frac{\partial}{\partial x_i} \left( \frac{\widetilde{b}_{I\cup
\{i\}}}{f^{\upsilon -1}} \right)-\frac{*}{f^{\nu-1}}.
\]
Put $\xi^\prime = \left ( */f^{\nu-1} : |I| = p \right )'$ and
$\delta = \left ( \widetilde{b}_{I\cup \{i\}}/f^{\nu-1} \mid i\not
\in I ,|I|=p \right )'$. Then we have
\[
\xi = \phi_{p+1}(\delta) + \xi^\prime.
\]
 So we have $x = [\xi] = [\xi^\prime]$. This yields $L(x) \leq L(\xi^\prime) \leq \nu -1$. This is a contradiction.

(b): Let $\nu = 1$ and $p \not = n-1.$ Note that $n\geq p+2.$ Also
note that $\xi \in (\mathbb{K}_p)_{-\omega}$. Thus for $|I|=p$ we
have
\[
\frac{a_I}{f}  \in (R_f(\omega_{i_1}+\cdots +\omega_{i_p})_{-\omega}
.
\]
 It follows that
 \[
 \deg a_I = \deg f - \omega + (\omega_{i_1}+ \cdots +\omega_{i_p}).
 \]
 Also note that $\deg \partial f/ \partial x_i = \deg f - \omega_i$.  By  comparing degrees in equation (\ref{t2-eq1}) we get
 $a_I = 0$ for all $I$ with $|I|=p$. Thus $\xi = 0$. So $x = 0.$ Therefore $u = 0$ a contradiction.

 (c): Let $p = n-1.$ By comparing degrees in equation (\ref{t2-eq1})
 we get $d_I = 0$ and $b_{I \cup \{i\}}=$constant. But

 \begin{align*}
\xi = \left (  \frac{\partial f}{\partial x_n}/f, -
\frac{\partial f}{\partial x_{n-1}}/f, \cdots, (-1)^{n-1}\frac{\partial
f}{\partial x_1}/f \right )'.
\end{align*}
It is easily verified that $\xi\in Z_{n-1}(\partial ; R_f)$ and that if $x=[\xi]$ then $\eta_1(x)=0.$

We prove that $\xi \not \in B_{n-1}(\partial ; R_f)$. Suppose if possible
  let $g\in R$, g.c.d$(g, f)=1$ and
\begin{align*}
 \quad \left (\frac{\partial}{\partial x_n},
-\frac{\partial}{\partial x_{n-1}}, \cdots ,
(-1)^{n-1}\frac{\partial}{\partial x_1} \right )'(g/f^c)&= \xi.\\
\end{align*}
 Thus \[\frac{\partial}{\partial x_i}(g/f^c)= \frac{\partial f}{\partial x_i}/f \quad \text{for} \quad i=1, \cdots, n.\]
 By computing left hand side we see that $f$ divides $g\partial f/\partial
 x_i$ for $i=1, \cdots , n.$
 \[\text{Let} \quad g\frac{\partial f}{\partial x_i} = fh_i.\quad  \]
 Let \[f= f_1^{a_1}f_2^{a_2}\cdots f_s^{a_s}, \quad f_j \ \text{irreducible and} \ a_j\geq1.\]

 Then \[\frac{\partial f}{\partial x_i} = \sum_{j=1}^s f_1^{a_1}\cdots f_{j-1}^{a_{j-1}}(a_jf_j^{a_j-1}\frac{\partial f_j}{\partial x_i})f_{j+1}^{a_{j+1}}\cdots f_s^{a_s}.\]

 If $a_j=1 \Rightarrow f_j $ does not divides $\partial f_j/\partial x_j$ so $f_j$ does not divides  $\partial f/\partial x_i.$

 If $a_j \geq 2 \Rightarrow f_j^{a_j-1}$ divides $\partial f/\partial x_i$ and $f_j^{a_j}$ does not divides $\partial f/\partial x_i.$ So we can write
 \[\frac{\partial f}{\partial x_i}= f_1^{a_1-1}f_2^{a_2-1}\cdots f_s^{a_s-1}V_i, \quad \text{where $f_j$ does not divides} \ V_i \ \forall \ j. \]
 Let $U=f_1f_2\cdots f_s.$ Then we have $gV_i=Uh_i.$ As g.c.d$(g, f)=1$ so g.c.d$(g, U)=1.$ So
 $f_j$ divides $V_i$ a contradiction.
  Therefore  $\xi \not \in B_{n-1}(\partial; R_f).$
 Hence $\ker(\eta_1) = K.$
 \end{proof}
By summarizing the above results, we have.
\begin{theorem}\label{H_p-filtration}
 Assume $H_{p+1}(\partial f; A)=0$. Then there exists a filtration $\{\mathcal{F}_{\nu}\}_{\nu \geq 0}$
consisting of $K-$ subspaces of $H_p(\partial ; R_f)$ with
$\mathcal{F}_{\nu} = H_p(\partial ; R_f)$ for $\nu \gg 0,
\mathcal{F}_{\nu} \supseteq \mathcal{F}_{\nu-1} $ and $\mathcal{F}_0
=0$ and  $K-$linear maps
\begin{align*}
\eta_{\upsilon}:\frac{\mathcal{F}_{\upsilon}}{\mathcal{F}_{\upsilon-1}}&
\rightarrow H_p(\partial f; A)_{(\upsilon + p)\deg f - \omega}.
\end{align*}
such that\\
(a) $\eta_{\upsilon}$ is injective for all $\upsilon \geq 2$.\\
(b) If $p\not= n-1.$ Then $\eta_1$ also injective.\\
(c) If $p=n-1.$ Then $\ker(\eta_1)= K.$
\end{theorem}

\begin{corollary}\label{H_i(R_f)-van}
 If $H_i(\partial f; A)=0$ for $i\geq \alpha +1$. Then
 \begin{align*}
 H_i(\partial ; R_f)=
\begin{cases} 0 & \text{if} \ \alpha +1 \leq i\leq n-2 \\ K & \text{if} \
i=n-1.
\end{cases}
\end{align*}
\end{corollary}
\begin{proof}
Let $\alpha +1 \leq i\leq n-2..$ By  Theorem \ref{H_p-filtration}
there exist a filtration $\{\mathcal{F}_{\upsilon}\}_{\upsilon \geq
0}$ of $H_i(\partial ; R_f)$ and injective maps
\begin{align*}
\eta_{\upsilon}:\frac{\mathcal{F}_{\upsilon}}{\mathcal{F}_{\upsilon
-1}}\longrightarrow H_i(\partial f; A).
\end{align*}
Note that $\mathcal{F}_0=0$. As $H_i(\partial f; A)=0$ we get
$\mathcal{F}_1=0.$ Continuing this way we get
$\mathcal{F}_{\upsilon}=0$ for all $\upsilon.$ As
$\mathcal{F}_{\upsilon}=H_i(\partial ; R_f)$ for $\upsilon \gg 0.$
Hence $H_i(\partial ; R_f)=0.$

Let $i=n-1.$ Then $\ker(\eta_1)=K.$ So $\mathcal{F}_1/K =0.$ Thus
$\mathcal{F}_1=K.$
\begin{align*}
\text{As} \quad  \dim_K H_{n-1}(\partial ; R_f) & = \sum
\dim_K(\mathcal{F}_{\upsilon}/\mathcal{F}_{\upsilon -1})\\
& = \dim_K \mathcal{F}_1\\
& =1.
\end{align*}
\end{proof}
We now have our main result.
\begin{theorem}\label{isolated-sing}
Let $f \in R$ be quasi homogeneous. Let $A=R/(f)$ be smooth. Then
 \begin{align*}
 H_i(\partial ; H^1_{(f)}(R))=
\begin{cases} 0 & \text{if} \quad   2 \leq i\leq n-2 \ \text{or} \ i=n\\ K & \text{if}
\quad i=n-1.
\end{cases}
\end{align*}
\end{theorem}
\begin{proof}
 As $A$ is smooth, so $H_i(\partial f ; A) = 0$ for
 $i\geq 2 $. Therefore by Corollary \ref{H_i(R_f)-van}
\begin{align*}
 H_i(\partial ; R_f)=
\begin{cases} 0 & \text{if} \ \alpha +1 \leq i\leq n-2 \\ K & \text{if} \
i=n-1.
\end{cases}
\end{align*}
 By \cite[Theorem 2.7]{TJ}
 \[H_n(\partial ;
H^1_{(f)}(R))=0 \quad  \text{and} \quad H_i(\partial ;
 R_f)\equiv H_i(\partial ;
H^1_{(f)}(R)) \  \text{for} \quad i<n.\] Hence  the result.
\end{proof}

\end{document}